\documentclass[11pt,a4paper,leqno]{article}
\usepackage[margin=1.1in]{geometry}
\usepackage{amsmath,amsthm,amssymb,enumerate,mathtools,color,float,textcomp, ragged2e}

\DeclarePairedDelimiter\ceil{\lceil}{\rceil}
\DeclarePairedDelimiter\floor{\lfloor}{\rfloor}
\definecolor{donkergroen}{RGB}{46,148,0}
\usepackage[hidelinks]{hyperref} 
\definecolor{donkerrood}{RGB}{204,0,0}

\definecolor{blauw}{RGB}{61,158,255}
\definecolor{donkerblauw}{RGB}{0,0,255}
\definecolor{donkergroen}{RGB}{46,148,0}
\definecolor{donkerrood}{RGB}{204,0,0}
\newenvironment{speciaalenumerate}{
\begin{enumerate}[(i)]
  \setlength{\itemsep}{1pt}
  \setlength{\parskip}{0pt}
  \setlength{\parsep}{0pt}
}{\end{enumerate}}

\makeatletter
\newif\if@borderstar
\def\bordermatrix{\@ifnextchar*{%
\@borderstartrue\@bordermatrix@i}{\@borderstarfalse\@bordermatrix@i*}%
}
\def\@bordermatrix@i*{\@ifnextchar[{\@bordermatrix@ii}{\@bordermatrix@ii[()]}}
\def\@bordermatrix@ii[#1]#2{%
\begingroup
\m@th\@tempdima8.75\p@\setbox\z@\vbox{%
\def\cr{\crcr\noalign{\kern 2\p@\global\let\cr\endline }}%
\ialign {$##$\hfil\kern 2\p@\kern\@tempdima & \thinspace %
\hfil $##$\hfil && \quad\hfil $##$\hfil\crcr\omit\strut %
\hfil\crcr\noalign{\kern -\baselineskip}#2\crcr\omit %
\strut\cr}}%
\setbox\tw@\vbox{\unvcopy\z@\global\setbox\@ne\lastbox}%
\setbox\tw@\hbox{\unhbox\@ne\unskip\global\setbox\@ne\lastbox}%
\setbox\tw@\hbox{%
$\kern\wd\@ne\kern -\@tempdima\left\@firstoftwo#1%
\if@borderstar\kern2pt\else\kern -\wd\@ne\fi%
\global\setbox\@ne\vbox{\box\@ne\if@borderstar\else\kern 2\p@\fi}%
\vcenter{\if@borderstar\else\kern -\ht\@ne\fi%
\unvbox\z@\kern-\if@borderstar2\fi\baselineskip}%
\if@borderstar\kern-2\@tempdima\kern2\p@\else\,\fi\right\@secondoftwo#1 $%
}\null \;\vbox{\kern\ht\@ne\box\tw@}%
\endgroup
}
\makeatother

\makeatletter 
\newcommand\mynobreakpar{\par\nobreak\@afterheading} 
\makeatother

\newcommand{\N}{\mathbb{N}}
\newcommand{\Z}{\mathbb{Z}}

\newcommand{\R}{\mathbb{R}}

\usepackage[english]{babel}

\newtheorem{theorem}{Theorem}[section]
\newtheorem{lemma}[theorem]{Lemma}
\newtheorem{claim}[theorem]{Claim}
\newtheorem{proposition}[theorem]{Proposition}
\newtheorem{corollary}[theorem]{Corollary}

\theoremstyle{definition}

\newtheorem*{examp*}{Example}

\def\lc{\left\lceil}   
\def\rc{\right\rceil}
\def\lf{\left\lfloor}   
\def\rf{\right\rfloor}
\theoremstyle{plain}

\makeatletter
\let\@fnsymbol\@arabic
\makeatother

\newfloat{Algorithm}{!hbt}{alg}

\newcounter{thm}[section]

\title{On the circular chromatic number of a subgraph of the Kneser graph}

\author{Bart Litjens\thanks{Korteweg-De Vries Institute for Mathematics, University of Amsterdam.
 Email: \texttt{bart\_litjens@hotmail.com}, \texttt{sven\_polak@hotmail.com}, \texttt{blsevenster@gmail.com}, \texttt{lluis.vena@gmail.com}. The research leading to these
results has received funding from the European Research Council under the European Union’s Seventh Framework Programme (FP7/2007-2013) / ERC grant agreement \textnumero 339109.} , Sven Polak\footnotemark[1] , Bart Sevenster\footnotemark[1] , Llu\' is Vena\footnotemark[1]}

\begin{document}

	\maketitle

\noindent \textbf{Abstract.} 
Let $n,k,r$ be positive integers with $n \geq rk$ and $r \geq 2$. Consider a circle $C$ with~$n$ points~$1,\linebreak[0]\ldots,\linebreak[0]n$ in clockwise order. The $r$-stable \emph{interlacing graph} $\text{IG}_{n,k}^{(r)}$ is the graph with vertices corresponding to $k$-subsets $S$ of $\{1,...,n\}$ such that any two distinct points in~$S$ have distance at least~$r$ around the circle, and edges between~$k$-subsets $P$ and $Q$ if they \emph{interlace}: after removing the points in~$P$ from $C$, the points in~$Q$ are in different connected components. In this paper we prove that the circular chromatic number of $\text{IG}_{n,k}^{(r)}$ is equal to $ n/k $ (hence the chromatic number is $\lceil n/k \rceil$) and that its circular clique number is also $ n/k $. Furthermore, we show that its independence number is $\binom{n-(r-1)k-1}{k-1}$, thereby strengthening a result by Talbot.\vspace{1mm}
\,$\phantom{0}$



\section{Introduction}

Let $n$ and $k$ be positive integers with $n \geq 2k$. The Kneser graph $\text{KG}_{n,k}$, first introduced by Martin Kneser in \cite{Kneser}, is the graph with vertices corresponding to $k$-subsets of $[n]:=\{1,\ldots,n\}$, where two vertices are adjacent if the corresponding sets are disjoint. Kneser conjectured that the chromatic number of $\text{KG}_{n,k}$ is $n-2k+2$. In \cite{Lovasz}, Lov\'asz proved this conjecture using topological methods. The Schrijver graph $\text{SG}_{n,k}$, also known as the $2$-stable Kneser graph, is the subgraph of $\text{KG}_{n,k}$ induced by those vertices that correspond to $k$-subsets of $[n]$ not containing adjacent elements in $[n]$ (here, $1$ and $n$ are adjacent). In~\cite{Schrijv}, Schrijver showed that $\text{SG}_{n,k}$ is a vertex-critical subgraph of $\text{KG}_{n,k}$ and also has chromatic number $n-2k+2$. (Vertex-critical means that the deletion of any vertex reduces the chromatic number.) Another famous result regarding the Kneser graph is the Erd\H{o}s-Ko-Rado theorem \cite{EKR}, which says that the maximum size of an independent set of $\text{KG}_{n,k}$ is $\binom{n-1}{k-1}$.

In this paper, all graphs are assumed to be finite. Let $G = (V,E)$ be a graph.
A \emph{circular coloring} of size $n/k$ is an assignment $\chi:V\to \mathbb{Z}/n\mathbb{Z}$, where $\Z/n\Z$ is the cyclic group of order $n$, such that $\chi(v_1)-\chi(v_2)\in\{\overline{k},\overline{k+1},\ldots,\overline{-k}\}$,
if $v_1v_2 \in E$. The \emph{circular chromatic number} $\chi_{c}(G)$ of $G$ is the minimal rational number $n/k$ for which there exists a circular coloring of size $n/k$. 
The \emph{circular clique} $K_{n/k}$ of size $n/k$ is the graph with vertex set $\Z/n\Z$, in which two vertices are adjacent if their distance is larger than or equal to $k$.  Hence, the circular chromatic number is  the minimal rational number $n/k$ for which there is a graph homomorphism of~$G$ to~$K_{n/k}$\footnote{Recall that a graph homomorphism from $G$ to $F$ is a map from the vertex set of $G$ to the vertex set of $F$ that maps edges of $G$ to edges of $F$.}. It is known that $\chi_{c}(K_{n/k})=n/k$ \cite{vince}.
 The \emph{circular clique number} $\omega_{c}(G)$ of $G$ is the maximal rational number~$n/k$ for which there is a homomorphism of~$K_{n/k}$ to~$G$. The numbers~$\chi_c(G)$ and~$\omega_{c}(G)$ are well-defined, as there is a homomorphism~$K_{n/k} \to K_{n'/k'}$ if and only if~$n/k\leq n'/k'$ \cite{starnote,vince} and the minimum and maximum in the definitions of the circular chromatic number and circular clique number respectively are indeed achieved~$\cite{vince,zhugt}$. Furthermore, it holds $\chi(G) = \lceil \chi_{c}(G)\rceil$ and~$\omega(G)=\floor{\omega_c(G)}$ (see, for example,~\cite{zhu}).
In \cite{chen}, Chen confirmed the conjecture from \cite{johholstah} that $\chi_{c}(\text{KG}_{n,k})=\chi(\text{KG}_{n,k})$, which was previously known for some cases, such as for even $n$ for the Schrijver graph \cite{meunier,simtar}.

In this paper we consider a subgraph of the Kneser graph. By abuse of notation, we sometimes write $P = \{1 \leq p_1 < \ldots < p_k \leq n\}$, if $P$ is a $k$-subset of $[n]$ consisting of the elements $p_1,\ldots,p_k$, with $p_1 < \ldots < p_k$. If $P$ and $Q$ are two $k$-subsets of $[n]$, with $P=\{1\leq p_1<\ldots<p_k\leq n\}$ and 
$Q=\{1\leq q_1<\ldots<q_k\leq n\}$, then 
$P$ and $Q$ are \emph{interlacing} if
either 
\[1\leq p_1<q_1<p_2<q_2<\cdots < p_k<q_k\leq n\]
or 
\[1\leq q_1<p_1<q_2<p_2<\cdots < q_k<p_k\leq n.\]
By distributing the elements of $[n]$ in clockwise order around a circle, we may view $P$ and $Q$ as $k$-polygons with points on the circle. Then $P$ and $Q$ are interlacing if 
removing the points of $P$ divides the circle into intervals that each contain one point of $Q$. We use this analogy to refer to $k$-subsets in $[n]$ as $k$-polygons, or just polygons when $k$ is understood.

We say that a polygon $P$ is $r$-\emph{stable} if for every pair of distinct points $a,b$ of $P$ there are at least $r-1$ elements of $[n]$ between $a$ and $b$ on the circle. The \emph{$r$-interlacing graph} $\text{IG}_{n,k}^{(r)}$ is the graph whose vertices correspond to $r$-stable $k$-polygons on $[n]$, and where two vertices are adjacent if the corresponding polygons are interlacing. Observe that polygons with two adjacent points would give rise to isolated vertices in the interlacing graph,  hence only $r$-stable polygons with $r\geq 2$ are considered. Note that $\text{IG}_{n,k}^{(r)}$ also is a subgraph of the Schrijver graph $\text{SG}_{n,k}$ and that $\text{IG}_{n,k}^{(2)}$ has the same vertex set as $\text{SG}_{n,k}$.

\subsection{Main results}

%
%
%
%
%

Fix positive integers $n,k,r$. If $kr> n$ then $\mathrm{IG}_{n,k}^{(r)}$ has no vertices. If $r=1$ then $\mathrm{IG}_{n,k}^{(r)}$ has isolated vertices. Hence, we will assume that $kr \leq n$ and $r \geq 2$ throughout this paper.
Our main result is the following.

\begin{theorem}\label{thm:circ_chrom}
	The circular chromatic number of $\mathrm{IG}_{n,k}^{(r)}$ is equal to $n/k$.
\end{theorem}

Thus, we immediately obtain the following corollary (cf.~\cite{zhu}). 

\begin{corollary}\label{thm:chrom}
	The chromatic number of $\mathrm{IG}_{n,k}^{(r)}$ is equal to $\lceil n/k \rceil$.
\end{corollary}

We also determine the independence number of the interlacing graph (Theorem~\ref{prop:size}), as well as the circular clique number (Theorem~\ref{thm:cliquenumber}).

\begin{theorem}\label{prop:size}
	The independence number of $\mathrm{IG}_{n,k}^{(r)}$ is $\binom{n-(r-1)k-1}{k-1}$. 
\end{theorem}

\begin{theorem}\label{thm:cliquenumber}
The circular clique number of $\mathrm{IG}_{n,k}^{(r)}$ is $n/k$.
\end{theorem}


The $2$-interlacing graph has connections with triangulations of cyclic polytopes. For $k=2$, note that two non-interlacing polygons on $[n]$ are just two non-crossing lines between vertices of an $n$-polygon. Then a maximal set of pairwise non-interlacing polygons is a triangulation. In \cite{Opper}, Oppermann and Thomas generalized this observation to higher dimensions: triangulations of the cyclic polytope with $n$ vertices in dimension $2k-2$, are in bijection with the independent sets of polygons in $\text{IG}_{n,k}^{(2)}$ of maximal size.
(The cyclic polytope $C(n,2k-2)$ with n vertices and dimension $2k-2$ is the convex hull of $n$ distinct points in $\R^{2k-2}$ that are obtained as evaluations of the curve defined by $P(x) = (x,x^2,\ldots,x^{2k-2})$, which is called the moment curve.)
%
In particular, the chromatic number of $\text{IG}_{n,k}^{(2)}$ gives the minimal size of a partition of the $(k-1)$-dimensional internal simplices of $C(n,2k-2)$ (see \cite{Opper}) in which no two simplices in each part internally intersect.

The proof of Theorem~\ref{prop:size} 
uses similar arguments as those in \cite{talbot2003intersecting}.
%
%
Theorem~3 in~\cite{talbot2003intersecting} can be interpreted as giving the independence number of the $r$-stable Kneser graph, where there is an edge between two $r$-stable sets (or polygons) if they are disjoint, and characterizes these maximal sets. Theorem~\ref{prop:size} strengthens \cite[Theorem~3]{talbot2003intersecting}, extending the upper bound to the spanning subgraph $\text{IG}^{(r)}_{n,k}$ of the $r$-stable Kneser graph. The coloring given in Section~\ref{s.low_bound_c} shows that the characterization of the maximum independent sets in \cite[Theorem~3]{talbot2003intersecting} does not extend to $\text{IG}^{(r)}_{n,k}$.

\section{The independence number: proof of Theorem~\ref{prop:size}}

For the proof of Theorem~\ref{prop:size} we follow the arguments of Talbot~\cite{talbot2003intersecting}. Let us begin by showing a preliminary lemma.

\begin{lemma} \label{lem:point}
The number of $r$-stable $k$-polygons on $[n]$ containing a specific point on the circle is $\binom{n-(r-1)k-1}{k-1}$. In particular, the number of vertices of $\mathrm{IG}_{n,k}^{(r)}$ is $\frac{n}{k} \binom{n-(r-1)k-1}{k-1}$.
\end{lemma}
\begin{proof} 
After fixing the point that all polygons have to contain, we see that we have to choose $k-1$ elements out of $n-1$ elements while respecting the minimum distance of $r$. So, we have to choose~$k$ distances~$a_1,\ldots,a_k$, all at least~$r$, such that~$\sum_{i=1}^k a_i =n$. The number of ways this can be done is 
\begin{align*}
|\{(a_1,\ldots,a_{k}) \in \Z_{\geq r}^{k} \mid \mbox{$\sum_{i=1}^k a_i = n$}\}| &= |\{(a_1,\ldots,a_{k}) \in \Z_{\geq 1}^k \mid \mbox{$\sum_{i=1}^k a_i = n-(r-1)k$}\}| \\ &=  \mbox{$\binom{n-(r-1)k-1}{k-1}$},
\end{align*}
proving the first statement. The second assertion follows from the first by a standard double counting argument.
\end{proof}

\begin{proof}[Proof of Theorem~\ref{prop:size}] The lower bound follows from Lemma~\ref{lem:point}. To show the upper bound, we use a double induction on $n$, and on $k$. The statement holds true for all~$n$ of the form~$n=kr$ with $k \in \N$, which are the base cases of the double induction, as in this case the graph is a complete graph on $r$ vertices $\{1,r+1,\ldots,(k-1)r+1\},\{2,r+2,\ldots,(k-1)r+2\},\ldots,\{r,2r,\ldots,kr\}$. 
Let $\binom{[n]}{k}^{(r)}$ denote the set of $r$-stable $k$-subsets of $[n]$, i.e., the vertex set of $\text{IG}^{(r)}_{n,k}$.

	Let 
	\[
	f:[n]\to [n-1], \hspace{2mm} f(i) := \left\{\begin{array}{ll} 1 & \text{if $i = 1$,}\\ i-1 & \text{otherwise,}\end{array}\right.
	\]
	be the function that fixes $1$ and shifts all the other elements counterclockwise.
	For a $k$-polygon $\{p_1,\ldots,p_k\}=P\in \binom{[n]}{k}$, we define $f(P):=\{f(p_1),\ldots,f(p_k)\}$. For a set $\mathcal{I}$ of $k$-polygons, we define $f(\mathcal{I}):=\{f(P)\}_{P\in \mathcal{I}}$. Applying $f$ several, say $r$ times, is denoted as $f^r$. 
	

We study the effect of $f$ on $k$-polygons.
Let $P\in \binom{[n]}{k}^{(r)}$. Then $f(P)\in \binom{[n]}{k}^{(r)}$, unless $1\in P$ and $r+1\in P$, in which case $f(P)$ is $(r-1)$-stable, but not $r$-stable.
Furthermore, $f(P)\in \binom{[n-1]}{k}^{(r)}$, unless $p_k-n< 1<p_1$ and $p_1+n-p_k= r$. In the latter case, $f(P)\in \binom{[n-1]}{k}$ is $(r-1)$-stable but not $r$-stable, since $f(p_k)<n<f(p_1)+n$  and $f(p_1)-f(p_k)+n-1=r-1$.

If $P,Q$ are $k$-polygons with $P\neq Q$ and $p_1\leq q_1$ and such that $f(P)=f(Q)$, then from the definition of $f$ it follows that $1\in P$, and $2\in Q$, and $p_i=q_i$ for all $i\in [2,k]$ (as otherwise either everything is left-shifted if $1\notin P$ or there is no collapse between $f(p_1)$ and $f(q_1)$ if $2\notin Q$). 

Any given set $\mathcal{I}\subseteq \binom{[n]}{k}^{(r)}$ can be partitioned into the following subsets:
\begin{itemize}
	\item $\mathcal{I}_{\text{int}}:=\{P\in \mathcal{I}\; |\; 1\in P \text{ and } \exists Q\in \mathcal{I} \text{ with } f(P)=f(Q)\}$. 
	\item $\mathcal{I}_{c}:=\{P\in \mathcal{I}\;|\; 1\in P \text{ and } r+1\in P\}$. 
	\item  $\mathcal{I}_{1}:=\{P\in \mathcal{I}\;|\; r\in P \text{ and } n\in P\}$.
	\item $\mathcal{I}_{2}:=\{P\in \mathcal{I}\;|\; r-1\in P \text{ and } n-1\in P\}$.
	\item $\ldots$
	\item $\mathcal{I}_{r-1}:=\{P\in \mathcal{I}\;|\; 2\in P \text{ and } n-r+2\in P\}$.
	\item $\mathcal{I}_{\text{rem}}:=\mathcal{I}\setminus \left(\mathcal{I}_{\text{int}}\sqcup \mathcal{I}_{c}\sqcup \left(\sqcup_{i=1}^{r-1} \mathcal{I}_{i}\right)\right)$. 
\end{itemize}
Indeed, any pair of the sets defined above are disjoint. For instance, $\mathcal{I}_{c}\cap \mathcal{I}_{\text{int}}=\emptyset$ since any element $Q\in \mathcal{I}$ with $f(Q)=f(P)$ for some $P\in \mathcal{I}_{\text{int}}$ is such that $r+1\notin Q$ (since $2\in Q$ and $Q$ is $r$-stable). The disjointness between any other pair of sets follows by $r$-stability, or by definition in the case of  $\mathcal{I}_{\text{rem}}$.

Notice that the sets $\mathcal{I}_{i}$ for $i\in[r-1]$ all contain $1$ strictly between two consecutive points at distance exactly $r$. 
 Let us also observe that if $Q\in\mathcal{I}$ is such that $2\in Q$ and there exists a $P\in\mathcal{I}$ with $f(Q)=f(P)$, then $Q\in \mathcal{I}_{\text{rem}}$ and $P\in \mathcal{I}_{\text{int}}$ (as the other option is that $Q\in \mathcal{I}_{r-1}$, which should be excluded as then $P$ would not be $r$-stable). 

Asssume that $\mathcal{I}$ is a set of $r$-stable $k$-polygons pairwise non-interlacing. Then:
\begin{enumerate}[\text{Claim} 1:]
	\item\label{en.1} We have $|f(\mathcal{I}_{\text{rem}})|=|\mathcal{I}_{\text{rem}}|$.
	\item\label{en.2} We have $f(\mathcal{I}_{\text{rem}})\subseteq \binom{[n-1]}{k}^{(r)}$.
	\item \label{en.25} Every pair of polygons in $f(\mathcal{I}_{\text{rem}})$ is pairwise non-interlacing.
	\item\label{en.3} We have $|f^{r-1}(\mathcal{I}_{\text{int}}\sqcup\mathcal{I}_{c} \sqcup \left(\sqcup_{i=1}^{r-1} \mathcal{I}_{i}\right))|=|\mathcal{I}_{\text{int}}\sqcup \mathcal{I}_{c} \sqcup \left(\sqcup_{i=1}^{r-1} \mathcal{I}_{i}\right)|$, where $\sqcup$ denotes the disjoint union.
	\item\label{en.4} 
 All the $k$-polygons in the set $f^{r-1}(\mathcal{I}_{\text{int}}\sqcup \mathcal{I}_{c}\sqcup \left(\sqcup_{i=1}^{r-1} \mathcal{I}_{i}\right))$ contain $1$. Upon removal of $1$ from these $k$-polygons, we can view them as elements in $\binom{[2,n-r+1]}{k-1}^{(r)}$. So
 	\[\left\{P\setminus \{1\}\; |\; P\in f^{r-1}(\mathcal{I}_{\text{int}}\sqcup\mathcal{I}_{c} \sqcup \left(\sqcup_{i=1}^{r-1} \mathcal{I}_{i}\right)\right\}\subseteq \binom{[2,n-r+1]}{k-1}^{(r)}\equiv \binom{[n-r]}{k-1}^{(r)}.\]
	\item  \label{en.5} Every pair of $(k-1)$-polygons in $\{P\setminus \{1\}\; |\; P\in f^{r-1}\left(\mathcal{I}_{\text{int}}\sqcup\mathcal{I}_{c} \sqcup \left(\sqcup_{i=1}^{r-1} \mathcal{I}_{i}\right)\right)\}$ is pairwise non-interlacing.
\end{enumerate}

Let us show the six claims, starting with Claim~\ref{en.1}. If $P$ and $Q$ are $k$-polygons for which $f(P) = f(Q)$, then, without loss of generality, $P \in \mathcal{I}_{\text{int}}$ and $Q \in \mathcal{I}_{\text{rem}}$. Additionally, for every $P \in \mathcal{I}_{\text{int}}$, there is a unique $k$-polygon $Q$ for which $f(P)=f(Q)$. Indeed, if $P$ is of the form $P=\{1= p_1<p_2<\cdots<p_k\}$ then $Q=\{2=q_1<p_2<\cdots<p_k\}$. This proves Claim~\ref{en.1}.

Claim \ref{en.2} follows as none of the $k$-polygons in the image of $\binom{[n]}{k}^{(r)}$ under $f$ contains $n$. Additionally, if $P\in \mathcal{I}_{\text{rem}}$, then $f(p_1)-(f(p_k)-n)\geq r+1$ (the $\Z_{n}$-cyclic distance between $p_1=1$ and $p_k$ is at least $r$, and $n$ lies between $f(p_k)$ and $f(p_1)+n$). Moreover, since $P\in \mathcal{I}_{\text{rem}}$, then $P\notin \mathcal{I}_{\text{c}}$ so $f(P)$ is $r$-stable. Therefore, we can view the image of $\mathcal{I}_{\text{rem}}\subset\binom{[n]}{k}^{(r)}$ under $f$ as a collection of $r$-stable $k$-polygons in $[n-1]$.

For Claim~\ref{en.25}, let $P$ and $Q$ be in $\mathcal{I}_{\text{rem}}$ and assume that $f(P)$ and $f(Q)$ interlace. Assume without loss of generality that $f(p_i)< f(q_i)<f(p_{i+1})<f(q_{i+1})$ for all $i\in[k-1]$. Then we have $p_i< q_i<p_{i+1}<q_{i+1}$ for all $i\in[k-1]$, as $f$ is non-decreasing, and hence $P$ and $Q$ would interlace. This proves Claim~\ref{en.25}.

Next, observe that $|f^{r-1}(\mathcal{I}_{\text{int}})|=|\mathcal{I}_{\text{int}}|$. Indeed, all polygons in $\mathcal{I}_{\text{int}}$ contain $1$, therefore their subsequent images are obtained by rotating the remaining $k-1$ points of the polygon $r-1$ steps counterclockwise, and thus two different $k$-polygons remain different. The same argument shows also that $|f^{r-1}(\mathcal{I}_{\text{c}})|=|\mathcal{I}_{\text{c}}|$, with the difference that every polygon in $f^{r-1}(\mathcal{I}_{\text{c}})$ contains $1$ and $2$ (and if $P,Q\in \mathcal{I}_{\text{c}}$ are different, they differ in a $j$-th point, with $j>2$, and this difference will be maintained in the $r-1$ rotations by $f$). The remaining cases $|f^{r-1}(\mathcal{I}_{\text{i}})|=|\mathcal{I}_{\text{i}}|$ for $i\in [r-1]$ follow similarly, as the polygons in $f^{r-1}(\mathcal{I}_{\text{i}})$ contain $1$ and $n-(r-1)-i+1$. 

To finish showing Claim~\ref{en.3} observe that the polygons in 
$f^{r-1}(\mathcal{I}_{\text{int}})$ contain $1$ but not $2$, and that their last element (when writing them in order) is less than or equal to $n-r+1-(r-1)$ (as the largest element in each $k$-polygon in $\mathcal{I}_{\text{int}}$ is less than or equal to $n-r+1$). The $k$-polygons in $f^{r-1}(\mathcal{I}_{\text{c}})$ contain $1$ and $2$ and their last element is less than or equal to $n-r+1-(r-1)$. The $k$-polygons in
$f^{r-1}(\mathcal{I}_{\text{i}})$ contain $1$ and $n-i+1-(r-1)$, which is such that $n-i+1-(r-1)> n-r+1-(r-1)$ as $i\in[1,r-1]$. Therefore, the images of all these sets by $f^{r-1}$ are pairwise disjoint and Claim~\ref{en.3} follows.

We have shown that each $k$-polygon in $f^{r-1}(\mathcal{I}_{\text{int}}\sqcup\mathcal{I}_{c} \sqcup \left(\sqcup_{i=1}^{r-1} \mathcal{I}_{i}\right))$ contains $1$. 
The other part follows by observing that applying the map $f$ to a $k$-polygon $r-1$ times and then removing $1$ from the resulting polygon is equivalent to first removing the minimal element of that polygon, then removing the next $r-1$ points to the left (in the cyclic order) in the underlying set (so $p_1,p_1-1,p_2-2,\ldots,p_1-r+1$, elements modulo $n$, are deleted) and then rotating the resulting polygon $r-1$ times, so that all the remaining $k-1$ points are in $[2,n-r+1]$. This proves Claim~\ref{en.4}.

Let us show Claim~\ref{en.5} via case analysis.
\begin{speciaalenumerate}
	\item If $P,Q\in \mathcal{I}_{i}$, with $i\in[r-1]$ then $f^{r-1}(P)\setminus \{1\}$ and $f^{r-1}(Q)\setminus \{1\}$ do not interlace as $n-i+1-(r-1)\in f^{r-1}(P)\setminus \{1\}$ and $n-i+1-(r-1)\in f^{r-1}(Q)\setminus \{1\}$.
		\item Similarly, if $P,Q\in \mathcal{I}_{\text{c}}$  then $2\in f^{r-1}(Q)\setminus \{1\}$ and $2\in f^{r-1}(P)\setminus \{1\}$, so $f^{r-1}(P)\setminus \{1\}$ and $f^{r-1}(Q)\setminus \{1\}$ do not interlace.
	\item Let $P\in \mathcal{I}_i$ and $Q\in \mathcal{I}_j$ with $i<j\leq r-1$, and assume $f^{r-1}(P)\setminus \{1\}$ and $f^{r-1}(Q)\setminus \{1\}$ interlace. Since $p_k<q_k$, then $f^{r-1}(p_k)<f^{r-1}(q_k)$. Therefore $f^{r-1}(p_2)<f^{r-1}(q_2)<\ldots<f^{r-1}(p_k)<f^{r-1}(q_k)$, so $p_2<q_2<\ldots<p_k<q_k$ as $f^{r-1}(x)$ is increasing in $x\in[r,n]$. Therefore 
	\[
	\{r-i+1<r-j+1<p_2<q_2<\ldots<p_k<q_k\}=\{p_1<q_1<p_2<q_2<\ldots<p_k<q_k\},
	\] so $P$ and $Q$ interlace in $\mathcal{I}$, a contradiction.

	\item Let $P\in \mathcal{I}_i$ and $Q\in \mathcal{I}_{\text{c}}$ with $i\in[r-1]$ and assume $f^{r-1}(P)\setminus \{1\}$ and $f^{r-1}(Q)\setminus \{1\}$ interlace. Then $2\in f^{r-1}(Q)\setminus \{1\}$. Hence, $f^{r-1}(q_2)=2<f^{r-1}(p_2)<\ldots<f^{r-1}(q_k)<f^{r-1}(p_k)$. As $f$ is increasing,
	\begin{align}
	q_2&<p_2<\ldots<q_k<p_k  \,\,\, \Longrightarrow \,\,\, \nonumber \\ \{1<r-j+1<r+1<p_2&<\ldots<q_k<p_k \}=\{q_1<p_1<q_2<p_2<\ldots<q_k<p_k \},\nonumber
	\end{align}
	so $P$ and $Q$ interlace in $\mathcal{I}$, a contradiction.
	
	\item Let $P\in \mathcal{I}_{\text{int}}$ and $Q\in \mathcal{I}_{\text{c}}$ and assume $f^{r-1}(P)\setminus \{1\}$ and $f^{r-1}(Q)\setminus \{1\}$ interlace.
	Then $2\in f^{r-1}(Q)\setminus \{1\}$, so $2\notin f^{r-1}(P)\setminus \{1\}$ and thus $f^{r-1}( p_2)>2$. Therefore,
	\[
	2=f^{r-1}(q_2)< f^{r-1}( p_2)<\ldots <f^{r-1}(q_k)< f^{r-1}( p_k)  \,\,\, \Longrightarrow \,\,\,
	q_2<  p_2<\ldots <q_k< p_k.
	\]
	Let $P'\in \mathcal{I}$ be the polygon for which $f(P)=f(P')$. Then
	$\{1<2<q_2<  p_2'<\ldots <q_k< p_k'\}=\{q_1<p_1'<q_2<  p_2'<\ldots <q_k< p_k'\}$.
	
	The case $P\in\mathcal{I}_{\text{int}}$ and $Q\in \mathcal{I}_{\text{i}}$ with $i\in [r-1]$ follows similarly but by directly using $P$ instead of $P'$ and noticing that, in that case, $1=p_1<q_1<\ldots<p_k<q_k$ since $p_k\leq n-r+1$ while $q_k=n-i+1\geq n-r+2$ (so $q_k>p_k$), and $q_1$ satisfies $r+1> q_1=r-i+1\geq 2$.
	
	\item Let $P,Q\in \mathcal{I}_{\text{int}}$ and assume that $f^{r-1}(P)\setminus \{1\}$ and $f^{r-1}(Q)\setminus \{1\}$ interlace. Assume without loss of generality that 
	$f^{r-1}(q_2)< f^{r-1}( p_2)<\ldots <f^{r-1}(q_k)< f^{r-1}( p_k)$. Then 
	$q_2<  p_2<\ldots <q_k<  p_k$. Let $P'$ be the polygon in $\mathcal{I}$ with $f(P)=f(P')$. Then $q_2<  p_2'<\ldots <q_k<  p_k'$ as $p_i'=p_i$ for $i\in[2,k]$. Then $1<2<q_2<  p_2'<\ldots <q_k<  p_k'$ so
	\[
	\{1<2<q_2<  p_2'<\ldots <q_k<  p_k'\}=\{q_1<p_1'<q_2<  p_2'<\ldots <q_k<  p_k'\}
	\] since $r\geq 2$. Thus $P'$ and $Q$ interlace, which is a contradiction.
\end{speciaalenumerate}
Let us finish the argument to show the proposition. Using the induction hypothesis we have
\[|f(\mathcal{I}_{\text{res}})|\leq \binom{n-1-(r-1)k-1}{k-1},\]
 as $f(\mathcal{I}_{\text{res}})$ is a set of mutually non-interlacing polygons in $\binom{[n-1]}{k}^{(r)}$ by Claim~\ref{en.2} and Claim~\ref{en.25}. Therefore 
 \[|\mathcal{I}_{\text{res}}|\leq \binom{n-1-(r-1)k-1}{k-1}\] as $|\mathcal{I}_{\text{res}}|=|f(\mathcal{I}_{\text{res}})|$ by Claim~\ref{en.1}.
Moreover, 
\[\{P\setminus \{1\}\;|\; P\in f^{r-1}(\mathcal{I}_{\text{int}}\sqcup\mathcal{I}_{\text{c}} \sqcup \left(\sqcup_{i=1}^{r-1} \mathcal{I}_{i}\right))\}\subset \binom{[2,n-r+1]}{k-1}^{(r)}\]
 by Claim~\ref{en.4} and contains mutually non-interlacing pairs of $k-1$-polygons by Claim~\ref{en.5}. Then, by the induction hypothesis 
\begin{align}
|f^{r-1}(\mathcal{I}_{\text{int}}\sqcup\mathcal{I}_{\text{c}} \sqcup \left(\sqcup_{i=1}^{r-1} \mathcal{I}_{i}\right))|
&\stackrel{\text{Claim}~\ref{en.4}}{=}|\{P\setminus \{1\}\;|\; P\in f^{r-1}(\mathcal{I}_{\text{int}}\sqcup\mathcal{I}_{\text{c}} \sqcup \left(\sqcup_{i=1}^{r-1} \mathcal{I}_{i}\right))\}|\nonumber \\
&\stackrel{\text{I.H.}}{\leq} \binom{n-r-(r-1)(k-1)-1}{k-2}.\nonumber \end{align}
 Utilizing Claim~\ref{en.3}, 
\[|\mathcal{I}_{\text{int}}\sqcup\mathcal{I}_{\text{c}} \sqcup \left(\sqcup_{i=1}^{r-1} \mathcal{I}_{i}\right)|\leq \binom{n-r-(r-1)(k-1)-1}{k-2}=\binom{n-(r-1)k-2}{k-2}.\]
Finally, using that $\mathcal{I}=\mathcal{I}_{\text{res}}\sqcup\mathcal{I}_{\text{int}}\sqcup\mathcal{I}_{\text{c}} \sqcup \left(\sqcup_{i=1}^{r-1} \mathcal{I}_{i}\right)$ we obtain
\begin{align*}
|\mathcal{I}|&=|\mathcal{I}_{\text{res}}|+|\mathcal{I}_{\text{int}}\sqcup\mathcal{I}_{\text{c}}\sqcup \left(\sqcup_{i=1}^{r-1} \mathcal{I}_{i}\right)|\\
&\leq \binom{n-1-(r-1)k-1}{k-1}+\binom{n-(r-1)k-2}{k-2}\\
&=\binom{n-(r-1)k-1}{k-1}
\end{align*}
as desired.
\end{proof}

\section{The circular chromatic number: proof of Theorem~\ref{thm:chrom}}

\subsection{Lower bound for the circular chromatic number}

We determine the lower bound of the circular chromatic number for the interlacing graph.

%

\begin{lemma}\label{cor:circ_low}
We have $\chi_{\mathrm{c}}\left(\mathrm{IG}_{n,k}^{(r)}\right)\geq n/k$.
\end{lemma}
\begin{proof}
The circular chromatic number $\chi_{c}(G)$ of a finite graph $G$ is lower bounded by the fractional chromatic number $\chi_f(G)$. Moreover, the fractional chromatic number $\chi_f(G)$ is bounded from below by the number of vertices of $G$ divided by the independence number of $G$ \cite[page 30]{ScheiUll11}. These two have been determined respectively in Theorem~\ref{prop:size} and in Lemma~\ref{lem:point} and their ratio is $n/k$. Hence the result is shown.
\end{proof}

\subsection{A circular coloring matching the lower bound}\label{s.low_bound_c}

To give a valid circular coloring of size $n/k$, we begin by presenting the following auxiliary technical lemma and its consequence, Corollary~\ref{cor_alt}.

%
%
%
%

\begin{lemma} \label{lem:perm_alt}
Let $y_1,\ldots,y_k \in \mathbb{R}_{\geq 0}$ and let $\sum_{i=1}^k y_i = z$. Then there exists a $j_0\in [k]$ such that for all $m \in [k]$ we have $\sum_{i = j_0}^{j_0+m-1} y_{i} \geq mz/k$, where the indices are taken modulo $k$. Moreover, either there exists an $m' \in [k]$ for which $\sum_{i = j_0}^{j_0+m'-1} y_{i} > m'z/k$, or $y_i=z/k$ for each $i\in[k]$.
\end{lemma}
\begin{proof}
Let $t$ be the number of $y_i$ that are unequal to $0$. We show the result with an induction on $t$. If $t$ is $0$ or $1$, then there is nothing to prove. Now assume $t > 1$ and suppose the result is true for all smaller $t$. Let $y_{q_1},...,y_{q_t}$ be the non-zero elements of $y_1,...,y_k$ with $q_1 < ... < q_t$. For $i = 1,...,t-1$, let $n_i = q_{i+1} - q_i$, and set $n_t = q_1-q_t+k$ (i.e., $n_i-1$ is the number of zeros between $q_i$ and $q_{i+1}$). Then $\sum_{i =1}^tn_iz/k = z$ and $\sum_{i=1}^t y_{q_i} = z$, so there exists an $l \in [t]$ such that $y_{q_l} \geq n_lz/k$. 
Now let $x_1,...,x_k$ be defined as $x_{q_l} = y_{q_l} + y_{q_{l+1}}$, $x_{q_{l+1}} = 0$ and $x_i = y_i$ for all other indices. 
Clearly, $t-1$ elements of the $x_i$ are unequal to $0$, so there exists some $j_0 \in [k]$ such that 
\[
\sum_{i = j_0}^{j_0 + m -1} x_i \geq mz/k,
\]
for all $m \in [k]$. Now notice that this implies the first statement of the lemma, namely that
\[
\sum_{i = j_0}^{j_0 + m -1} y_i \geq mz/k, 
\]
for all $m \in [k]$. \\
\indent The second assertion of the lemma follows from the first since if for all $m\in[k]$, both $\sum_{i = j_0}^{j_0+m-1} y_i=mz/k$ and $\sum_{i = j_0}^{j_0+(m+1)-1} y_i=(m+1)z/k$, then $y_{j_0+m}=z/k$.
\end{proof}

%

For a $k$-polygon $P = \{1 \leq y_1 < ... < y_k \leq n\}$, define the $k$-tuple $s(P)$ of distances between the consecutive points by
\[
s(P):=(y_2-y_1,y_3-y_2,\ldots,y_1-y_k+n) \in \Z_{\geq 1}^k.
\]
We call $s(P)$ the \textit{shape} of $P$. We say that the $k$-polygon with points $y_1+1,...,,y_k+1$ is obtained from $P$ by a clockwise rotation of $1$ (the addition is modulo $n$). Similarly, for $i \geq 0$ the $k$-polygon obtained by rotating clockwise $i$ times is denoted by $\rho_i(P)$.

\begin{corollary} \label{cor_alt}
Let $P = \{1 \leq y_1 < ... < y_k \leq n\}$ be a $k$-polygon and write $s(P) = (d_1,...,d_k)$ for the shape of $P$. Then there is a~$j_0 \in [k]$ such that for~$i' = n- y_{j_0}$ we have
\begin{align}\label{condities}
n \in \rho_{i'}(P), \; &\text{and }\; |\rho_{i'}(P) \cap \{1,\ldots,\lc m n/k \rc-1\}| < m,
\end{align}
for all $m \in [k]$. Thus, for $m \in [k]$,
\begin{align}\label{condities_2}
|\rho_{i'}(P) \cap \{1,\ldots,\lf m n/k \rf\}| = m &\Longleftrightarrow \lf m n/k \rf=mn/k \text{ and } \sum_{i=j_0}^{j_0+m-1}d_i=mn/k,\\
\label{condities_3}
|\rho_{i'}(P) \cap \{1,\ldots,\lc m n/k \rc\}| = m 
&\Longrightarrow \lc m n/k \rc\in \rho_{i'}(P).
\end{align}
\end{corollary}
\begin{proof}
We start with assertion~\eqref{condities}. Note that $\sum_{i=1}^{k}d_i = n$. By Lemma~\ref{lem:perm_alt}, there exists a $j_0$ such that $\sum_{i = j_0}^{j_0+m-1} d_{i} \geq mn/k$ for every $m\in[k]$. Let $0 \leq i' < n$ be the unique natural number such that $y_{j_0}+i' = n$, i.e., such that $\rho_{i'}$ rotates the $j_0$-th point of $P$ to the point $n$. We claim that this choice of $i'$ satisfies the conditions in (\ref{condities}). 
The rotated polygon $\rho_{i'}(P)$ is the $k$-polygon $\{1 \leq y_{j_0+1}+i' < ... < y_{j_0}+i' = n\}$. Hence $s(\rho_{i'}(P)) = (d_{j_0+1},...,d_{j_0+k-1},d_{j_0})$, and
\begin{equation}
\sum_{i = j_0}^{j_0+m-1} d_{i}=y_{j_0+m}+i'\geq mn/k,
\end{equation}
for all $m \in [k]$. Hence there exists a~$t \in [k]$ such that $\rho_{i'}(P) \cap \{1,...,\ceil{mn/k}-1\}=\{y_{j_0+1}+i',\ldots,y_{j_0+t}+i'\}$. Using that $y_{j_0+m}+i'\geq mn/k$, so $y_{j_0+m}+i'\geq \lc mn/k\rc$ as $y_{j}$ are integers (as $d_i$ are positive integers), this implies that
\begin{equation}\label{ongelijk}
|\rho_{i'}(P) \cap \{1,...,\ceil{mn/k}-1\}|=|\{y_{j_0+1}+i',\ldots,y_{j_0+t}+i'\}| < m,
\end{equation} 
as desired.


Part \eqref{condities_3} follows immediately. To prove \eqref{condities_2}, note that $\ceil{mn/k}-1=\floor{mn/k}$ unless $mn/k$ is an integer. Furthermore, if $|\rho_{i'}(P) \cap \{1,...,\lf mn/k \rf\}|=m$, by the just proven result, $mn/k=\lf mn/k \rf=y_{j_0+m}+i'$ and so $\sum_{i = j_0}^{j_0+m-1} d_i = y_{j_0+m}+i'= mn/k$ as claimed.
%
%
%
%
%
\end{proof}

Let us notice that $\text{IG}^{(r)}_{n,k}$ is an induced subgraph of $\text{IG}^{(2)}_{n,k}$ for any $r\geq 2$ (or containing no vertices if $n<kr$). Therefore, Theorem~\ref{thm:circ_chrom} is shown if we give an $n/k$ circular coloring to $\text{IG}_{n,k}=\text{IG}^{(2)}_{n,k}$.

We now use the previous results to find large independent sets in the interlacing graph $\textrm{IG}_{n,k}$, that eventually will be color classes of a particular circular coloring.
For a vector $d = (d_1,..,d_k) \in \Z_{\geq 2}^k$ with $\sum_{i=1}^{k}d_i = n$, let $P_d^{\circ}$ be a $k$-polygon with $s(P_d^{\circ}) = d$, such that $P_d^{\circ}$ contains the point $n$, and such that $|P_d^{\circ}\cap \{1,...,\lceil mn/k\rceil -1 \}| < m$, for each $m \in [k]$.
Corollary \ref{cor_alt} guarantees the existence of such a $k$-polygon (then~$j_0=1$ in the corollary). 
The set of $k$-polygons of the form $P_d^{\circ}$ (for some $d$ as above) is an independent set in $\text{IG}_{n,k}$ (as all polygons contain the point $n$). Define
\[
\mathcal{L}_{n,k} := \left\{ P^{\circ}_d  \, \bigg| \, d = (d_1,...,d_k) \in \Z_{\geq 2}^k  \text{ and } \sum_{i=1}^t d_i \geq tn/k \text{ for all } t \in [k]\right\}
\] as the set of such polygons.
The next lemma summarizes the main properties of the polygons in $\mathcal{L}_{n,k}$.

\begin{lemma}\label{lem:independent_rotations}
	For any $i,j \in \Z_{\geq 0}$, the sets 
\[
\rho_{j}(\mathcal{L}_{n,k})\hspace{0mm} \cup \hspace{0mm} \rho_{j+\lfloor i n/k \rfloor} (\mathcal{L}_{n,k}) \hspace{2mm}\mathrm{and}\hspace{2mm} \rho_j(\mathcal{L}_{n,k}) \hspace{0mm}\cup\hspace{0mm} \rho_{j+\lceil i n/k \rceil} (\mathcal{L}_{n,k})
\]
are independent sets in $\mathrm{IG}_{n,k}$, where
	$\rho_{t} (\mathcal{L}_{n,k}):=\{\rho_t(Q)\;|\; Q\in \mathcal{L}_{n,k}\}$, for $t \in \Z_{\geq 0}$.
\end{lemma}
\begin{proof}
We may assume that $j=0$, as given any $\lambda\geq 0$: $P\in \rho_{j}(\mathcal{L}_{n,k})$ and $Q\in 
\rho_{j+\lambda} (\mathcal{L}_{n,k})$ are adjacent in $\textrm{IG}_{n,k}$ if and only if $\rho_{-j}(P)\in \mathcal{L}_{n,k}$ and $\rho_{-j}(Q)\in \rho_{\lambda}(\mathcal{L}_{n,k})$ are adjacent as well.

Let $P\in \mathcal{L}_{n,k}$, $Q_1\in \rho_{\lfloor i n/k \rfloor} (\mathcal{L}_{n,k})$, and $Q_2\in \rho_{\lceil i n/k \rceil} (\mathcal{L}_{n,k})$. By Corollary~\ref{cor_alt} and the fact that $n\in P$, either
	\begin{itemize}
		\item $P$ contains at most $i$ points between $n$ and $\lfloor i n/k \rfloor$, or
		\item $P$ contains exactly $i+1$ points between 
	$n$ and $\lfloor i n/k \rfloor$, and then $\lfloor i n/k \rfloor=in/k$  and $in/k\in P$ (the second part of the corollary).
	\end{itemize}
	Additionally, either 
	\begin{itemize}
		\item $P$ contains at most $i$ point between $n$ and $\lceil i n/k \rceil$, or
		\item $P$ contains exactly $i+1$ points between 
	$n$ and $\lceil i n/k \rceil$, and $\lceil i n/k \rceil\in P$.
	\end{itemize}
	By Corollary~\ref{cor_alt}, the polygon $Q_1$ contains at least $i+1$ points
	 between $n$
	  and $\lfloor in/k\rfloor$.
\footnote{$Q_1$ is the rotation of a polygon $Q_1'$ that contains either \begin{itemize}
	 \item  at least $i+1$ points in $[\lceil (k-i) n/k\rceil,n]$, so 
	 $Q_1$ would contain $\geq i+1$ points in $[\lceil (k-i) n/k\rceil +1+\lfloor i n/k\rfloor,\lfloor i n/k\rfloor]\mod n \subseteq [n,\lfloor i n/k\rfloor]$, or 
	 \item exactly $i$ points in $[\lceil (k-i) n/k\rceil +1,n]$, and $\lceil (k-i) n/k\rceil\in Q_1'$, so $Q_1$
 contains $i$ points  in $[\lceil (k-i) n/k\rceil +1 +\lfloor i n/k\rfloor +1,\lfloor i n/k\rfloor] \subseteq [n +1,\lfloor i n/k\rfloor]\mod n$ and $n\in Q_1$, so it contains $i+1$ points in $[n,\lfloor i n/k\rfloor]$.
\end{itemize}} Thus, either
	\begin{itemize}
		\item $P$ contains $i$ points in the interval and $Q_1$ at least $i+1$, in which case they are not adjacent by a pigeonhole argument, or
		\item $P$ contains $i+1$ points in the interval and $\lfloor in/k\rfloor\in P$, and, since $\lfloor in/k\rfloor\in Q_1$ as well, they are also not adjacent.
	\end{itemize}
 Similarly, by Corollary~\ref{cor_alt}, the polygon $Q_2$ contains at least $i+1$ points between $n$ and $\lceil i n/k \rceil$.\footnote{$Q_2$ is the rotation of a polygon $Q_2'$ that contains either \begin{itemize}
			\item  at least $i+1$ points in $[\lfloor (k-i) n/k\rfloor +1,n]$, so 
			$Q_2$ would contain $\geq i+1$ points in $[\lfloor (k-i) n/k\rfloor +1+\lceil i n/k\rceil,\lfloor i n/k\rfloor]\mod n=[n+1,\lfloor i n/k\rfloor]\mod n \subseteq [n,\lfloor i n/k\rfloor]$, or 
			\item exactly $i$ points in $[\lfloor (k-i) n/k\rfloor +1,n]=[\lfloor (k-i) n/k\rfloor +1,n]$, and $\lfloor (k-i) n/k\rfloor=(k-i) n/k$, and $(k-i) n/k\in Q_2'$, so $Q_2$
			contains $i$ points  in $[\lfloor (k-i) n/k\rfloor+\lfloor i n/k\rfloor +1,\lfloor i n/k\rfloor]\subseteq[n +1,\lfloor i n/k\rfloor]\mod n$ and $n\in Q_2$, so $Q_2$ contains $i+1$ points in $[n,\lfloor i n/k\rfloor]$.
	\end{itemize}} Thus, either
	\begin{itemize}
		\item $P$ contains $i$ points in the interval and $Q_2$ at least $i+1$, thus by a pigeonhole argument they are not adjacent, or
		\item $P$ contains $i+1$ points in the interval and $\lceil in/k\rceil\in P$, and, since $\lceil in/k\rceil\in Q_2$  contains $i+1$ points in this interval as well, in which case they are also not adjacent.
	\end{itemize}
Therefore, we conclude that neither $P$ and $Q_1$, nor $P$ and $Q_2$ are adjacent. This finishes the proof.
\end{proof}

\begin{proof}[Proof of Theorem~\ref{thm:circ_chrom}]
Lemma~\ref{cor:circ_low} provides the lower bound. Hence, it remains to show the upper bound. As mentioned earlier, it suffices to find an $n/k$-circular coloring for $\text{IG}_{n,k}=\text{IG}^{(2)}_{n,k}$.

By Corollary~\ref{cor_alt}, any polygon~$P$ is the rotation of a polygon in $\mathcal{L}_{n,k}$. Define a circular coloring $\chi: V \rightarrow \Z/n\Z$ (where $V$ denotes the set of vertices of $\textrm{IG}_{n,k}$) by~$\chi(P)=\overline{ik}$, where $0 \leq i \leq n-1$ is the smallest number such that $P\in \rho_{i} (\mathcal{L}_{n,k})$. 

We will characterize for which $0 \leq s \leq n-1$ we have $\overline{sk} \notin \{\overline{k},\overline{k+1},...,\overline{-k}\}$. Let~$i \in \N$. If~$s=\lceil in/k \rceil$, then~$sk \in \{in,\ldots,in+k-1\}$, hence $\overline{sk} \in \{\overline{0},\ldots,\overline{k-1}\}$. Moreover, if~$s=\lfloor (i+1)n/k\rfloor$, then~$sk \in \{(i+1)n-k+1,\ldots,(i+1)n\}$, hence $\overline{sk} \in \{\overline{-(k-1)},\overline{-(k-2)},\ldots,\overline{-1},\overline{0}\}$. So if~$ \lfloor in/k \rfloor<s< \lfloor (i+1)n/k \rfloor$, then~$sk \in \{k+in, \ldots, n-k+in\}$, hence $\overline{sk} \in \{\overline{k},\ldots, \overline{-k}\}$. It follows that for $0 \leq s \leq n-1$ we have~$\overline{sk} \notin \{\overline{k},\ldots,\overline{-k}\}$ if and only if there is an~$i \in [k]$ with $s \in\{\lfloor in/k\rfloor,\lceil in/k\rceil\}$.

Let $P$ and $Q$ be polygons, with $\chi(P) = \overline{ik}$ and $\chi(Q) = \overline{jk}$. By the observation above, $\chi(Q)-\chi(P)\not\in \{\overline{k},\ldots,\overline{-k}\}$ if and only if $j=i+\lceil tn/k\rceil$ or $j=i+\lfloor tn/k \rfloor$, for some $t\in[k]$.
By Lemma~\ref{lem:independent_rotations}, if $P\in \rho_{i}(\mathcal{L}_{n,k})$ and, $Q\in \rho_{i+\lceil tn/k\rceil}(\mathcal{L}_{n,k})$ or $Q\in \rho_{i+\lfloor tn/k \rfloor}(\mathcal{L}_{n,k})$, then
$P$ and $Q$ are not adjacent. Therefore, the coloring $\chi$ is a valid circular coloring of size $n/k$. This establishes the upper bound and finishes the proof of the theorem.
\end{proof}

\section{The circular clique number: proof of Theorem~\ref{thm:cliquenumber}}

In this section, we prove Theorem~\ref{thm:cliquenumber}, giving the circular clique number of $\mathrm{IG}_{n,k}^{(r)}$. 

Let $n',k'$ be such that $n'/k' = n/k$ and $\mathrm{gcd}(n',k') = 1$. By Theorem~\ref{thm:circ_chrom}, the circular clique number of $\mathrm{IG}_{n,k}^{(r)}$ is at most $n'/k'$. Hence, to prove Theorem~\ref{thm:cliquenumber}, it remains to find a circular clique of the appropriate size in $\mathrm{IG}_{n,k}^{(r)}$. This is the content of the following proposition.

\begin{proposition} \label{lem:lower_bound}
	Let $n',k'$ be such that $\gcd(n',k')=1$ and $n/k=n'/k'$.
	The subgraph $G_{\mathrm{equi}}$ of $\mathrm{IG}^{(r)}_{n,k}$ induced by the $n'$ different $r$-stable polygons $\{P^j\}_{0 \leq j \leq n-1}$, where
	\begin{equation}\label{set_p}
	P^j:=\{j+n,j+\lceil n/k\rceil,j+\lceil 2n/k\rceil,\ldots,j+\lceil i n/k\rceil,\ldots,j+\lceil (k-1) n/k\rceil\},
	\end{equation}
	is a circular clique of size $n'/k'$.
	
	%
	%
	%
	%
\end{proposition}
\begin{proof}
	
Observe that $P^j=P^i$ if and only if $\overline{i} = \overline{j}$ in $\Z/(n/\gcd(n,k))\Z$. Indeed, $in/k\in \Z$, for $i\in [k]$, if and only if $i=
\lambda k/\gcd(n,k)$ for some $\lambda$ such that $\lambda k/\gcd(n,k)\in[k]$. Furthermore, the integers $\{\lceil in/k\rceil\}_{i\in [k/\gcd(n,k)]}$ are uniquely determined by the pair $(n',k')$ with $\gcd(n',k')=1$ and $n'/k'=n/k$ (or the rational number $n'/k'$).
Hence,
the set of polygons \eqref{set_p} has size $n'$, which shows the first assertion of the lemma.

Without loss of generality, we can thus assume that $\gcd(n,k)=1$. In the other cases, the rotations of the polygons in the following arguments have a multiplicative factor of $\gcd(n,k)$, so the rotated polygons considered are the ones corresponding to elements in the subgroup $\gcd(n,k)\Z/n\Z$ of $\Z/n\Z$ of size $n'$.

Recall that $K_{n/k}$ is the graph with the elements of $\Z/n\Z$ as vertices and two vertices being adjacent if their difference is in $\{\overline{k},...,\overline{-k}\}$ (larger or equal than $k$).


We define a graph homomorphism from $G_{\textrm{equi}} \rightarrow K_{n/k}$ by  
\begin{equation}\label{eq.v_assign}
P^j \mapsto \overline{jk}.
\end{equation}
Since $\gcd(n,k)=1$, the map \eqref{eq.v_assign} is a bijection. We now show it is a graph homomorphism.

Without loss of generality, it suffices to argue about the adjacencies involving $P:=P^0$. 
By Lemma~\ref{lem:independent_rotations}, the polygon $P$ does not interlace with the polygons
$P^t$, with $t\in \cup_{i\in[k]} \{\lfloor in/k\rfloor, \lceil in/k\rceil\}$, and those correspond precisely (using the vertex assignment \eqref{eq.v_assign}) to the vertices in $K_{n/k}$ that are at distance strictly less than $k$ from $\overline{0}$.

It remains to show that $Q :=P^{\lceil in/k\rceil +s}$ and $P$ interlace, for each $0 \leq i \leq k-1$ and each 
\begin{align}\label{interval}
s\in [1,\lfloor (i+1)n/k\rfloor - \lceil i n/k\rceil-1].
\end{align}
If the interval in $(\ref{interval})$ is empty, we are done. So suppose that the interval is nonempty. Note that the polygon~$Q$ consists of the points $\lceil in/k\rceil +s,\lceil in/k\rceil +s+\lceil n/k\rceil,
\lceil in/k\rceil +s+\lceil 2n/k\rceil,\ldots,\lceil in/k\rceil +s+\lceil (k-1)n/k\rceil$.

Observe that $Q$ and $P$ interlace
if for any $0 \leq j \leq k-1$, the $j$-th point in $Q$ lies between $\lceil(i+j)n/k\rceil$ and
$\lceil(i+j+1)n/k\rceil$ (which are consecutive points in $P$). That is,~$Q$ and~$P$ interlace if for each $0 \leq j \leq k-1$ we have
\begin{align} \label{interlacingsituation}
\lceil(i+j)n/k\rceil   <  \ceil{in/k}+s+\ceil{jn/k} < \lceil(i+j+1)n/k\rceil.
\end{align}
Let $\{x\}:=x-\lfloor x\rfloor$ denote the fractional part of $x$. Then in general,
\begin{align}\label{eq.up1}
	&\lceil in/k\rceil + \lceil jn/k\rceil = 
\begin{cases}
\lceil (i+j)n/k\rceil & \text{ if } \left\{in/k\right\}+\left\{jn/k\right\}>1 \text{ or } \left\{in/k\right\}\cdot\left\{jn/k\right\}=0,\\
\lceil (i+j)n/k\rceil +1 & \text{ if } \left\{in/k\right\}+\left\{jn/k\right\} \leq 1\text{ and }\left\{in/k\right\}\cdot\left\{jn/k\right\}\neq 0, \\
\end{cases} \\ \label{eq.up2}
&\lfloor in/k\rfloor + \lceil jn/k\rceil = 
\begin{cases}
\lceil (i+j)n/k\rceil & \text{ if } \left\{in/k\right\}+\left\{jn/k\right\}\leq 1 \text{ and } \left\{jn/k\right\}\neq 0,\\
\lfloor (i+j)n/k\rfloor & \text{ if } \left\{in/k\right\}+\left\{jn/k\right\} \geq 1 \text{ or } \left\{jn/k\right\}=0.\\
\end{cases}
\end{align}

Suppose first that~$s=1$.  Using \eqref{eq.up1}, we obtain that~$(\ref{interlacingsituation})$ is satisfied unless 
$\lceil (i+j)n/k\rceil +2=\lceil (i+j+1)n/k\rceil$. This implies that $\lfloor n/k\rfloor=2$ (since $n\geq 2k$ by assumption), that $\left\{in/k\right\}+\left\{jn/k\right\} \leq 1$, and that $in/k,jn/k\notin \Z$.
%
%



Assume that $\{in/k\}+\{n/k\}\geq 1$. Since 
\[
\lceil (i+j)n/k\rceil + \lfloor n/k\rfloor = \lceil (i+j)n/k\rceil + 2=\lceil (i+j+1)n/k\rceil,
\]
by \eqref{eq.up2} we have 
\[
\{(i+j)n/k\}+\{n/k\}\leq 1 \text{ and } \{(i+j)n/k\}\neq 0.
\]
This implies that $\{in/k\}+\{jn/k\} <1$ (so $\{(i+j)n/k\}=\{in/k\}+\{jn/k\}$). Hence, $\{in/k\}+\{jn/k\} + \{n/k\}\leq 1$, which contradicts the fact that $\{in/k\}+\{n/k\}\geq 1$, as $jn/k\notin \Z$.

Therefore, we have $\{in/k\}+\{n/k\}< 1$. Hence,  $(i+1)n/k\notin \Z$ and, by \eqref{eq.up2},
 \begin{align*}
\lceil in/k\rceil + 1 &= \lceil in/k\rceil + 2-1=\lceil in/k\rceil + \lfloor n/k\rfloor -1=\lceil (i+1)n/k\rceil -1
\\&=\lfloor (i+1)n/k\rfloor+1 -1=\lfloor (i+1)n/k\rfloor.
\end{align*}
So the interval in~$(\ref{interval})$ is empty, a contradiction. Hence~$(\ref{interlacingsituation})$ holds for~$s=1$.

%
%
%
%

It now suffices to show that, for any fixed~$0 \leq j \leq k-1$, equation~$(\ref{interlacingsituation})$ also holds with the largest~$s$ from~$(\ref{interval})$. As we already know that~$\eqref{interlacingsituation}$ holds for~$s=1$, we then have~$(\ref{interlacingsituation})$ for all~$s$ in the interval~$(\ref{interval})$. For~$s=\lfloor (i+1)n/k\rfloor - \lceil i n/k\rceil-1$ one obtains
\begin{align*}
\ceil{in/k}+s+\ceil{jn/k} & = \ceil{in/k}+\lfloor (i+1)n/k\rfloor - \lceil i n/k\rceil-1+\ceil{jn/k}
\\&= \lfloor (i+1)n/k\rfloor +\ceil{jn/k}-1 .
\end{align*}
However, by \eqref{eq.up2}, the last expression is equal to $\lceil(i+j+1)n/k\rceil -1$ or $\lfloor(i+j+1)n/k\rfloor -1$, which is in both cases strictly smaller than $\lceil(i+j+1)n/k\rceil$.  Moreover, $\ceil{in/k}+s+\ceil{jn/k} > \lceil (i+j)n/k\rceil$. Indeed, $s\geq 1$ as
\[\lfloor (i+1) n/k\rfloor-\lceil in/k\rceil -1\geq \lfloor n/k\rfloor
+\lceil in/k\rceil -\lceil in/k\rceil -1= \lfloor n/k\rfloor -1\geq 1,\]
and
$\lceil in/k\rceil+\lceil j n/k\rceil\geq \lceil (i+j)n/k\rceil$. This proves~$(\ref{interlacingsituation})$ for~$s=\lfloor (i+1)n/k\rfloor - \lceil i n/k\rceil-1$ and concludes the proof.
\end{proof}

\section*{Acknowledgments}

We would like to thank Lex Schrijver and Arnau Padrol for their useful comments.

\end{document}